\documentclass[a4paper, 12pt]{amsart}
\usepackage{cite}
\usepackage{amsmath}
\usepackage{amsthm}
\usepackage[all]{xy}
\usepackage{amssymb}
\usepackage{mathrsfs}
\usepackage{etoolbox}
\usepackage{fullpage}

\newtheorem{thm}{Theorem}
[section]
\newtheorem{cor}[thm]{Corollary}
\newtheorem{lem}[thm]{Lemma}
\newtheorem{prop}[thm]{Proposition}

\theoremstyle{definition}

\newtheorem{fct}[thm]{Fact}
\newtheorem{defn}[thm]{Definition}

\newtheorem{rmk}[thm]{Remark}
\numberwithin{equation}{subsection}

\usepackage{color}
\begin{document}

\title{On algebraic relations between solutions of a generic Painlev\'e equation.}
\author{Ronnie Nagloo$^1$, Anand Pillay$^2$ \\University of Leeds}
\date{\today}
\maketitle 
\pagestyle{plain}
\begin{abstract}
We prove that if $y'' = f(y,y',t, \alpha, \beta,..)$ is a ``generic" Painlev\'e equation (that is, an equation in one 
of the families $P_{I} - P_{VI}$  but with the relevant complex parameters $\alpha, \beta,..$ algebraically independent), and if $y_{1},...,y_{n}$ are  solutions such that $y_{1},y_{1}',y_{2},y_{2}',...,y_{n},y_{n}'$ are algebraically dependent over ${\mathbb C}(t)$, then already for some $1\leq i < j \leq n$, $y_{i},y_{i}',y_{j},y_{j}'$ are algebraically dependent over ${\mathbb C}(t)$. The proof combines results by the Japanese school on  ``irreducibility" of the 
Painlev\'e equations, with the trichotomy theorem for strongly minimal sets in differentially closed fields.
\end{abstract}
\footnotetext[1]{Supported by an EPSRC Project Studentship and a University of Leeds - School of Mathematic partial scholarship}
\footnotetext[2] {Supported by EPSRC grant EP/I002294/1}

\section{Introduction}
The Painlev\'e equations are  ordinary differential equations of order $2$ and come in six families $P_{I} - P_{VI}$, where 
$P_{I}$ consists of the single equation $y'' = 6y^{2} + t$, and $P_{II}-P_{VI}$ come with some complex parameters. They were  
isolated (or discovered) in the early part of the 20th century, by Painlev\'e, with refinements by Gambier and Fuchs, as those ODE's of 
the form $y'' = f(y,y',t)$ (where $f$ is rational over ${\mathbb C}$) which have the {\em Painlev\'e property} and define, via their solutions, ``new special functions", at least for general values of the associated complex parameters. Interest in these differential equations and their solutions has increased considerably since the 1970's.\\

In this paper we are concerned with algebraic relations or dependencies over ${\mathbb C}(t)$ between solutions of a ``general" such equation (namely an equation in one of the classes $P_{I}-P_{VI}$ with complex parameters in general position). We consider a set $y_{1},..,y_{n}$ of solutions as meromorphic functions on some disc $D\subseteq C$ and so working in the differential field $F$ of meromorphic functions on $D$ (which contains the differential subfield ${\mathbb C}(t)$ of rational functions) we can ask about $tr.deg({\mathbb C}(t)(y_{1},y_{1}',..,y_{n},y_{n}')/{\mathbb C}(t))$. We conjecture that this is always $2n$, namely $y_{1},y_{1}',..,y_{n},y_{n}'$ are {\em algebraically independent} over ${\mathbb C}(t)$. It has been brought to our notice that claims of this nature were made in past  by Drach, Vessiot and others.  In any case this ``conjecture" has only been definitively proved, quite recently by Nishioka \cite{Nishioka2}, in the case of $P_I$. Our main result, as stated in the abstract, is a {\em weak} version of this algebraic independence conjecture, but valid for 
{\em all} generic Painlev\'e equations. \\

We will discuss now more of the background and methods.
Much effort has gone into describing algebraic (over ${\mathbb C}(t)$) solutions of the Painlev\'e equations. For $P_{I} - P_{VI}$ there is a classification of the parameters for which the equation has algebraic solutions, and for such parameters a description of the algebraic solutions. We will take it as known that for generic equations in each family $P_{I}-P_{VI}$ there are 
{\em no} algebraic solutions, although strictly speaking this is not needed for our main result as stated in the abstract. \\

Considerable work has also gone into clarifying Painlev\'e's notion of {\em irreducibility}, and  a series of papers by Okamoto, Nishioka, Noumi, Umemura, Watanabe, and others have led to a proofs of  ``irreducibility" of $P_{I}-P_{VI}$ outside special values of the complex parameters. This special notion of irreducibility should not be confused with the usual notions of irreducibility of an algebraic variety (or ``differential algebraic variety"). We will discuss 
 irreducibility  in the appendix. But what was actually proved by the Japanese school, was that, except for some 
special values of the parameters, the set of solutions of a Painlev\'e equation, considered as a ``differential algebraic variety" or ``definable set 
in an ambient differentially closed field" is {\em strongly minimal} in the sense of model theory.\\

Maybe the deepest result in the 
model theory of differentially closed fields concerns the classification of strongly minimal sets $X$. This is originally due to Hrushovski and Sokolovic \cite{Hrushovski-Sokolovic} but their paper was never published, so we will mention other routes to or references for this result in section 2.
The trichotomy statement for strongly minimal sets in a differentially closed field says that 
there are 3 types. Type (I), ``nonorthogonality to the constants" is a version of {\em algebraic integrability} (after base change). Type (II) says 
that $X$ is closely related to the solution set $A^{\sharp}$ of a very special kind of ODE on a simple abelian variety $A$ or rather complex abelian 
scheme ${\mathbb A}$ (of which Painlev\'e-Picard is a special case). And Type (III), ``geometric triviality" is a general form of the algebraic dependence statement in the abstract. We are able to use the afore-mentioned results of the Japanese school, together with additional techniques to rule out Type (I) and Type (II), when $X$ is the solution set of a generic Painlev\'e equation, thus giving the desired conclusion. The additional methods include ``uniform definability of the $A\to A^{\sharp}$ functor, which is alluded to in \cite{Hrushovski-Itai}, and which we give a detailed account of in section 2. What is also crucial is that for each of the Painlev\'e families II-VI, the set of complex tuples $(\alpha, \beta,..)$ for which the corresponding equation is {\em not} strongly minimal has ``infinitely many components". \\

The proof of our main result is very straightforward modulo the ingredients mentioned above, and no detailed analysis of the relevant differential equations is required, other than quoting/translating  the results of the Japanese school. Much of this paper consists of an account of the various ingredients so as to be accessible to several audiences.
In section 2 we discuss the differential algebraic/model-theoretic context and the main notions and results relevant to this paper, including the trichotomy theorem. In section 3 we describe the results by the Japanese school on ``irreducibility" of the Painlev\'e equations $P_{I}-P_{VI}$, and in each case prove our main results (Propositions 3.1, 3.6, 3.9, 3.12, 3.15, and 3.18).  As mentioned earlier the appendix discusses the notion of {\em irreducibility} of an ODE.  It is not needed for our main results but is included for cultural reasons.\\

Below we take  the opportunity to list the families $P_{I}-P_{VI}$. But before doing so we should also say something about the {\em Painlev\'e property} $PP$, although this is not needed for our results: For simplicity consider an ODE 
$y^{(n)} = f(y,y',y'',..,y^{n-1},t)$ where $f$ is a rational function over ${\mathbb C}$. Let $S$ be the finite set of singularities of the equation including the point at infinity if necessary. 
Let $X$ be the universal cover of $P^{1}({\mathbb C})\setminus S$. Then $PP$ for the ODE above can be interpreted as saying that any local analytic solution extends to a meromorphic solution on $X$. (This interpretation was pointed out to us by Malgrange.) One question is whether the Painlev\'e property has a differential-algebraic description or characterization. This was verified for order $1$ equations  (\cite{vanderPut}, \cite{Matsuda}) but we do not know  about higher order equations.\\

We now list the families $P_{I}-P_{VI}$.

\begin{equation*}
\begin{split}
P_{I}:\;\;\;\;\;  & y'' =6y^2+t \\
P_{II}(\alpha):\;\;\;\;\; & y'' =2y^3+ty+\alpha \\
P_{III}(\alpha,\beta,\gamma,\delta):\;\;\;\;\; & y'' =\frac{1}{y}(y')^2 -\frac{1}{t}y'+\frac{1}{t}(\alpha y^2+\beta)+\gamma y^3+\frac{\delta}{y} \\
P_{IV}(\alpha,\beta):\;\;\;\;\; & y'' =\frac{1}{2y}(y')^2+\frac{3}{2}y^3+4ty^2+2(t^2-\alpha)y+\frac{\beta}{y} \\
P_{V}(\alpha,\beta,\gamma,\delta):\;\;\;\;\; & y'' =\left(\frac{1}{2y}+\frac{1}{y-1}\right)(y')^2-\frac{1}{t}y'+\frac{(y-1)^2}{t^2}\left(\alpha y+\frac{\beta}{y}\right)+\gamma\frac{y}{t}\\
& +\delta\frac{y(y+1)}{y-1} 
\end{split}
\end{equation*}
\begin{equation*}
\begin{split}
P_{VI}(\alpha,\beta,\gamma,\delta):\;\;\;\;\; & y'' =\frac{1}{2}\left(\frac{1}{y}+\frac{1}{y+1}+\frac{1}{y-t}\right)(y')^2-\left(\frac{1}{t}+\frac{1}{t-1}+\frac{1}{y-t}\right)y'\\
 & +\frac{y(y-1)(y-t)}{t^2(t-1)^2}\left(\alpha+\beta\frac{t}{y^2}+\gamma\frac{t-1}{(y-1)^2}+\delta\frac{t(t-1)}{(y-t)^2}\right)
\end{split}
\end{equation*} \\


Finally in this introduction some acknowledgements: The paper \cite{Hrushovski-Itai} is a big influence, for various reasons discussed subsequently. A preliminary version of that paper also discussed the issue of strong minimality of a generic $P_{II}$ equation, but without definitive conclusions. Thanks are due  to Dave Marker for his helpful comments following a talk by the second author on this material in the Kolchin seminar (New York, October 2011). Thanks also to Frank Nijhoff and Davide Penazzi for numerous discussions related to the material presented here.

\section{Differential algebraic geometry and the model theory of differentially closed fields}

Kolchin's ``Differential Algebraic Geometry" is one particular abstract formalism for studying algebraic differential 
equations. The set-up is analogous to that of Weil's Foundations in algebraic geometry, and for that reason may not be 
too fashionable nowadays. However the point of view coincides with that of the model theory of differentially closed fields: 
Kolchin's ``universal domain" 
$({\mathcal U},+,\cdot,\partial)$ is a saturated model of the first order complete theory 
$DCF_{0}$, and the basic objects of study, ``differential algebraic varieties", locally modelled on subsets of ${\mathcal U}^{n}$ cut out by finitely many {\em differential polynomial equations}, coincide (up to finite Boolean 
combination) with the {\em definable sets} in $\mathcal U$. We refer the general reader to \cite{Marker-book} and \cite{Pillay-lecture notes} for the basics of model theory, to \cite{Pillay-Newton} for a survey of model theory and its applications (where differential fields are central), and to the article \cite{Marker}  for more details on the model theory of differentially closed fields.  For general stability theory and ``geometric" stability theory, which is relevant to subsections 2.1, 2.3, 2.4, 3.7, and the appendix, the reader is referred to \cite{Pillay-book}. \\

We  quickly recall some key notions. Fix the first order language $L$ of differential rings: the language of unitary rings $\{+,-,\cdot, 0,1\}$ together with a symbol $\partial$ for the derivation. The theory $DCF_{0}$, in the language $L$, consists of the axioms for fields of characteristic $0$ together with another infinite collection of axioms, asserting, of a differential field $K$, that any system of differential polynomial equations over $K$ (in finitely many differential indeterminates) with a common solution in some differential field extension of $K$ already has a common solution 
in $K$. The theory $DCF_{0}$ is complete, $\omega$-stable, and has quantifier elimination, and the models of $DCF_{0}$ are called the differentially closed fields (of characteristic $0$). Note that a differentially closed field is algebraically closed (by considering the special case of axioms which concern only polynomial equations). As a matter of notation, for a field $F$, $F^{alg}$ denotes its algebraic closure (in the usual algebraic sense). A remark is that if $(F,\partial)$ is a {\em differential} field (of characteristic $0$) then $\partial$ extends uniquely to a derivation on $F^{alg}$.

In analogy with the Zariski topology from algebraic geometry, if $K$ is a differentially closed field  (i.e. model of $DCF_{0}$) and $X\subseteq K^{n}$ we call $X$ {\em Kolchin closed} if $X$ is the common zero set of finitely many differential polynomials $F(y_{1},..,y_{n})$ with coefficients from $K$. Such a differential polynomial is simply a usual polynomial over $K$ in indeterminates 
\newline 
$y_{1},..,y_{n},\partial(y_{1}),...,\partial(y_{1}), \partial^{2}(y_{1}),..,\partial^{2}(y_{n}),....$. 
\newline
Sometimes we call a Kolchin closed subset of $K^{n}$ an {\em affine differential algebraic variety}. There is a theory of (abstract) differential algebraic varieties, which are modelled locally on affine differential algebraic varieties  (for example \cite{Kolchin},\cite{Buium}) but there is no need for us to elaborate on this here. 
In any case quantifier elimination for $DCF_{0}$ implies that for a differentially closed field $K$, the {\em definable} (even with parameters) subsets of $K^{n}$ coincide with the finite Boolean combinations of Kolchin closed subsets of $K^{n}$ (i.e. finite unions of ``locally" Kolchin closed subsets of $K^{n}$). If a Kolchin closed set $X$ (or more generally definable set) is defined by a system of differential polynomials ($L$-formulas) with parameters from a differential subfield $k$ of $K$ we will say $X$ is {\em defined over $k$} or just {\em over $k$}. \\

We will fix a ``saturated" model $\mathcal U = ({\mathcal U},+,-,\cdot,0,1,\partial)$ of $DCF_{0}$ which for convenience we will assume to be of cardinality continuum. Hence the (algebraically closed) field of constants ${\mathcal C}_{\mathcal U}$ of $\mathcal U$ can be identified with the field of complex numbers ${\mathbb C}$. For $y\in {\mathcal U}$ we may write $y'$ for $\partial(y)$, likewise for a finite tuple $y = (y_{1},..,y_{n})$. We should make it clear that $\mathbb C$ can be seen in two ways, as a subfield of ${\mathcal U}$ (over which varieties, differential algebraic varieties,..  may be defined) and as a differential algebraic variety (or definable set) in its own right, defined by $y' = 0$.

We let $t$ denote an element of $\mathcal U$ such that $t' = 1$. Hence the differential field $({\mathbb C}(t),d/dt)$ is a differential subfield of ${\mathcal U}$. If $D\subset {\mathbb C}$ is an open connected domain then then the set 
${\mathcal F}(D)$ of meromorphic functions on $D$, equipped with $d/dt$ will be a differential field extension of 
$({\mathbb C}(t),d/dt)$. Clearly ${\mathcal F}(D)$ has cardinality $>$ the continuum so can not be embedded in ${\mathcal U}$, but if $F$ is any differential subfield of  ${\mathcal F}(D)$ containing and countably generated over ${\mathbb C}(t)$ it {\em will be} embeddable in ${\mathcal U}$ over ${\mathbb C}(t)$.\\

If $F$ is a differential subfield of ${\mathcal U}$ and $y$ is a tuple from ${\mathcal U}$, then by $F\langle y \rangle$
we mean the differential subfield of $\mathcal U$ generated by $F$ and $y$. It clearly coincides with the field $F(y,y',y'',...)$ generated by $F$ together with $y,y',y'',...$. With this notation, and given a tuple $z$ from ${\mathcal U}$ we will say  ``$z\in acl(F,y)$"  if the coordinates of $z$ are in $F\langle y \rangle^{alg}$. What we have described as ``$acl(F,y)$" is precisely the {\em algebraic closure} of $F,y$ in the sense of model theory in the structure $({\mathcal U},+,\cdot,\partial)$. \\

Given a definable set $Y\subseteq {\mathcal U}^{n}$ (or as a special case, a differential algebraic variety), defined over a countable differential 
subfield $F$, we can consider $sup\{tr.deg(F\langle y \rangle/F):y\in Y\}$ as $y$ ranges over elements of $Y$. If this number is finite, then we call it the 
{\em order} of $Y$ and say that $Y$ is {\em finite-dimensional}, which is equivalent to $Y$ having finite Morley rank.
Assuming $order(Y)$ to be {\em finite} we will call a point $y\in Y$ {\em generic over $F$} if $tr.deg(F\langle y\rangle/F) = order(Y)$. 
\\

Recall:

\begin{defn} A definable set $X$ in ${\mathcal U}$ is {\em strongly minimal} if it is infinite but has no infinite co-infinite definable subsets.
\end{defn}

It is not hard to see that a strongly minimal set must be finite-dimensional. If $X\subseteq {\mathcal U}^{n}$ is a differential algebraic variety then $X$ being strongly minimal is equivalent (via quantifier elimination) to $X$ having no proper infinite differential algebraic subvariety.\\

Definition 2.1 (as a definition of strong minimality) makes sense for a definable set in any saturated structure, and we will freely use this definition later. 

\begin{rmk} Let $X$ be a definable set of finite order $m$ say. Then the following are equivalent:
\newline
(i) $X$ is strongly minimal,
\newline
(ii) $X$ can not be written as the disjoint union of definable sets of order $m$, {\em and} for any differential field $F$ over which $X$ is defined, and point $y\in X$, either $y\in F^{alg}$ or $y$ is generic point of $X$ over $F$ (i.e. 
$tr.deg(F\langle y \rangle/F)$ is $0$ or $m$). 
\end{rmk}

Note the special case:

\begin{rmk} Let $Y\subseteq {\mathcal U}$ be defined by differential equation $y^{(n)} = f(y,y',..,y^{(n-1)},t) = 0$ where $f$ is rational over $\mathbb C$. Then $Y$ is strongly minimal if and only if for any differential subfield $F$ of $\mathcal U$ which is finitely generated over ${\mathbb C}(t)$, and $y\in Y$, either $y\in F^{alg}$  or $tr.deg(F\langle y \rangle)/F) = n$. 
\end{rmk}

We mention a few examples:
\newline
- If $f$ is an absolutely irreducible polynomial over $\mathcal U$ in $2$ variables then the subset $Y$ of ${\mathcal U}$ defined by $f(y,y') = 0$ is strongly minimal, of order $1$.
\newline
- If $Y\subseteq {\mathbb C}^{n}$ is an algebraic variety, then note that $Y$ as a subset of ${\mathcal U}^{n}$ is a differential algebraic variety, defined by the relevant polynomial equations over $\mathbb C$ together with differential equations $y_{1}' = y_{2}' = .. = y_{n}' = 0$. In particular
when $Y\subset {\mathbb C}^{n}$ is an {\em irreducible algebraic curve} then as a {\em differential algebraic variety}, $Y$ is strongly minimal, of order $1$. 
\newline
- The subset of $\mathcal U$ defined by  $\{yy'' = y', y' \neq 0\}$ is strongly minimal, of order $2$. (See 5.17 of \cite{Marker}.)

\begin{rmk} Suppose $X$ is strongly minimal, defined over $F$, and of order $m$. Let $y_{1},..,y_{k}\in X$. Then 
$tr.deg(F\langle y_{1},..,y_{k}\rangle/F) = mi$ for some $i=0,..,k$. 
\end{rmk}

\vspace{2mm}
\noindent
Some other examples, important for this paper, come from simple abelian varieties $A$  over ${\mathcal U}$. 
We will identify $A$ with its set $A({\mathcal U})$ of ${\mathcal U}$-points and so can consider differential algebraic (equivalently definable) subgroups of $A$. 
\begin{fct} ($A$ an abelian variety over ${\mathcal U}$.) 
\newline
(i) $A$ has a (unique) smallest Zariski-dense definable subgroup, which we call $A^{\sharp}$. 
\newline
(ii) $A^{\sharp}$ is finite-dimensional, $dim(A)\leq order(A^{\sharp}) \leq 2dim(A)$ and moreover $dim(A) = order(A^{\sharp})$ if and only if $A$ descends to ${\mathbb C}$ (in which case $A^{\sharp} = A({\mathbb C})$), 
\newline
(iii) If $A$ is a simple abelian variety with ${\mathbb C}$-trace $0$, then $A^{\sharp}$ is strongly minimal. 
\newline
(iv) If $A$ is an elliptic curve then $A^{\sharp}$ is strongly minimal, whether or not $A$ descends to ${\mathcal C}$.
\end{fct}
\begin{proof} (i) and (ii) are due to Alie Buium \cite{Buium} with references. (iii) depends on the dichotomy theorem for strongly minimal sets in $DCF_{0}$ and will be discussed below. (iv) follows from the last part of (ii), together with (iii). 
\end{proof}

\vspace{2mm}
\noindent
\begin{defn} Let $Y\subset {\mathcal U}^{n}$ be strongly minimal, and suppose $order(Y) = m$.
\newline
(i) We call $Y$ {\em geometrically trivial} if for some (any) countable differential field $F$ over which $Y$ is defined, if 
$y_{1},..,y_{\ell}$ are elements of $Y$, generic over $F$, and $tr.deg(F\langle y_{1}, y_{2},..,y_{\ell} \rangle/F) < \ell m$ then for some $i<j$,
\newline
 $tr.deg(F\langle y_{i}, y_{j} \rangle /F) < 2m$ (equivalently = $m$). Another way of expressing this is that for any countable differential field $F$ over which $Y$ is defined, and for any $y_{1},..,y_{\ell}\in Y$, if the collection consisting of $y_{1},..,y_{\ell}$ together with all their derivatives $y_{i}^{(j)}$ is algebraically dependent over $F$ then for some $i<j$, $y_{i}, y_{j}$ together with their derivatives is algebraically dependent over $F$. (So geometric triviality is precisely the statement in the abstract.)
\newline
(ii) Let $Z$ be another strongly minimal set. We will say that $Y$ and $Z$ are nonorthogonal if there is some definable relation $R\subset Y\times X$ such that the images of the projections of $R$ to $Y,Z$ respectively are infinite (hence cofinite in $Y, Z$ respectively), and these projections are ``finite-to 1".  
\end{defn}

\begin{rmk} If $Y,Z$ are nonorthogonal strongly minimal sets then $order(Y) = order(Z)$. Moreover $Z$ is geometrically trivial if and only if $Y$ is.
\end{rmk}

We can now state the trichotomy theorem.
\begin{thm} Let $X$ be a strongly minimal set. Then exactly one of the following holds:
\newline
(i) $X$ is nonorthogonal to the strongly minimal set $\mathbb C$  (defined by $y' = 0$),
\newline
(ii) $X$ is nonorthogonal to $A^{\sharp}$ for some simple abelian variety $A$ over ${\mathcal U}$ which does not descend to ${\mathbb C}$. 
\newline
(iii) $X$ is geometrically trivial. 
\end{thm}

\subsection{Guide to the proof of the trichotomy theorem.}

Everything at the model-theoretic level that one needs to know for the proof of our main result is given above, but we will here try to give a guide through the literature for the interested reader. The first key result is the {\em dichotomy theorem} for strongly minimal sets in $DCF_{0}$. It requires the notion of {\em modularity} of a strongly minimal set. First the heuristics: given a strongly minimal set $X$, we can consider the definable set $X\times X$ and strongly minimal subsets of $X\times X$. Modularity of $X$ means that for a {\em generic} point $(a,b)$ of $X\times X$ there is no infinite definable family of strongly minimal subsets of $X\times X$ passing through $(a,b)$. (And in this kind of language, {\em geometric triviality} of $X$ is equivalent to the stronger property of there being no infinite definable family of strongly minimal subsets of $X\times X$  at all.)\\

Now for the formalities. We work with $DCF_{0}$. We use the model-theoretic notion of algebraic closure given above. Let $X$ be a strongly minimal set.
\begin{defn} $X$ is said to be {\em modular} if for some (countable) differential field $F$ over which $X$ is defined, the following holds:
\newline
if $B\subset X$, $c, d\in X$ and $d\in acl(F,B,c)$ then there is $b\in X$ such that $b\in acl(F,B)$ and $d\in acl(F,b,c)$. 
\end{defn}

\begin{rmk} If $X$ is strongly minimal and geometrically trivial then $X$ is modular. Moreover modularity is invariant under nonorthogonality.
\end{rmk}

Consider  algebraically independent points $a,b,c\in {\mathbb C}$, let $d = ac + b$  and note that $d\in acl(a,b,c)$  but there is no $b'\in acl(a,b)$ such that $d\in acl(b',c)$. So  $\mathbb C$ is nonmodular. \\

The dichotomy theorem in the context of $DCF_{0}$ is:

\begin{thm} Suppose $X$ is strongly minimal. Then either $X$ is modular or $X$  is nonorthogonal to the strongly minimal set ${\mathbb C}$ (and not both).
\end{thm}

The first proof of Theorem 2.11 was given in \cite{Hrushovski-Sokolovic}. It went through the notion of a {\em Zariski geometry}. We will not give the definition of Zariski geometry here but just mention that it is closely related to the notion of an {\em algebraic $\partial$-variety} discussed below. In any case, one first shows that (after removing finitely many points) $X$ can be assumed to be a Zariski geometry. Nonmodularity of $X$ together with the main theorem of \cite{Hrushovski-Zilber} yields a definable strongly minimal field $k$ with is nonorthogonal to $X$. The theory of ``finite-dimensional" groups definable in ${\mathcal U}$ then gives that $k$ must be definably isomorphic to the field ${\mathcal C}_{\mathcal U}$ of constants. 

A second, and more direct proof of Theorem 2.11, avoiding the dependence on the deep and difficult Zariski geometries theorem from \cite{Hrushovski-Zilber}, was given in \cite{Pillay-Ziegler}. Making use of ``differential jet spaces" of differential algebraic varieties (which generalize Kolchin's differential tangent spaces), it was shown that nonmodularity of $X$ yields ``nonorthogonality" of $X$ to a differential jet space $J(X)_{a}$ at a generic point $a$ of $X$. But $J(X)_{a}$ is the solution space of a linear differential equation  hence is a finite-dimensional vector space over ${\mathbb C}$, whereby $X$ is nonorthogonal to $\mathbb C$.\\

To pass from Theorem 2.11 to the Trichotomy Theorem 2.7 requires an understanding of modular strongly minimal sets in $DCF_{0}$ and uses a combination of general geometric stability theoretic results with results specific to $DCF_{0}$. The arguments are given in \cite{Pillay-DAG}, in the paragraphs leading up to Proposition 4.10 there, and we repeat/summarize them here. Firstly for a modular strongly minimal set $X$ in {\em any} structure, either $X$ is geometrically trivial or $X$ is nonorthogonal to a definable modular strongly minimal group. This is due to Hrushovski 
\cite{Hrushovskilocallymodular} and also is treated in \cite{Pillay-book} (see Theorem 1.1 of Chapter 5). So we may assume $X$ to be a strongly minimal modular group $G$ (which has to be commutative, either by strong minimality or modularity). As discussed in \cite{Pillay-DAG} $G$ definably embeds in a connected commutative algebraic group $A$ without proper connected positive dimensional algebraic subgroups. So $A$ is either the additive group, the multiplicative group, or a simple abelian variety. In the first two cases our understanding of their differential algebraic subgroups yields that $G$ is nonorthogonal to ${\mathcal C}_{\mathcal U}$, contradiction, so $A$ must be a simple abelian variety. Strong minimality of $G$ forces $G$ to be $A^{\sharp}$ from Fact 2.4(i), and by Fact 2.4 (ii) $A$ does not descend to ${\mathbb C}$ (for then $A^{\sharp}= A({\mathbb C})$ and $G$ is nonorthogonal to 
${\mathcal C}_{\mathcal U}$, contradiction again). \\

This essentially completes the  guide to Theorem 2.7. We will now return as promised to part (iii) of Fact 2.5. and also discuss the mutual exclusivity of conditions (i), (ii), (iii) in 2.8. So let $A$ be a simple abelian variety over $\mathcal U$ which does not descend to $\mathbb C$. We want to see that $A^{\sharp}$ is strongly minimal {\em and modular}. This is also explained on p.130 of \cite{Pillay-DAG} and we repeat some things. First, simplicity of $A$ and Fact 2.5 (i) imply that $A^{\sharp}$ is ``minimal" in the sense of having no proper definable infinite subgroup. Let $X$ be a strongly minimal definable subset of $A^{\sharp}$. If $X$ were nonmodular then the dichotomy theorem together with some additional arguments yield that $A$ descends to $\mathbb C$, contradiction. So $X$ is modular. The minimality of $A^{\sharp}$ as mentioned above implies that $A^{\sharp}$ is contained in $acl(X)$ (together with finitely many additional parameters). The modularity of $X$ yields that $A^{\sharp}$ is a so-called ``$1$-based group" which implies in particular that any definable subset of $A^{\sharp}$ is a translate of a subgroup (up to finite Boolean combination). The ``minimality" of $A^{\sharp}$ as defined above implies that $A^{\sharp}$ has NO infinite co-infinite definable subset, so is strongly minimal, and of course modular.\\

 This modularity of $A^{\sharp}$ implies that (i) and (ii) in Theorem 2.8 are mutually exclusive. On the other hand  if $G$ is a strongly minimal group defined over $F$, and $a,b$ are ``mutually generic" elements of $G$, then putting $c = a\cdot b$, the triple $\{a,b,c\}$ is a counterexample to the geometric triviality of $G$. So we see that (ii) and (iii) are also mutually exclusive.\\
 
 This completes our commentary.
 
 \subsection{Algebraic $\partial$-varieties}
 
 It will be useful, especially with regard to the the material in section 3, to recall the notion of an algebraic $\partial$-variety. This appears in \cite{Buium} and \cite{Hrushovski-Itai} for example, as well as in \cite{Pillay-tworemarks} where the relationship with Zariski structures is made explicit. In any case this is just another formalism for algebraic differential equations, and when the base field of definition is the function field of a curve, reduces to the notion of a connection in the sense of Ehresmann. 
 
\begin{defn} Let $V$ be an algebraic variety, not necessarily irreducible, defined over a differential field $(F,\partial)$. A structure on $V$ of an {\em algebraic $\partial$-variety over $F$} is given by an extension $D_{V}$ of $\partial$ to a derivation of the structure sheaf ${\mathcal O}_V$.We call the pair $(V,D_{V})$ an algebraic $\partial$-variety over $F$.
\end{defn}

\begin{rmk} (i) If $V$ is affine irreducible, then $D_{V}$ as above is simply a derivation of the coordinate ring $F[V]$ of $V$ which extends $\partial$. 
\newline
(ii) In the same way as a derivation on ${\mathcal O}_{V}$ extending the trivial derivation on $F$ corresponds to a regular section of the tangent bundle $T(V)$ of $V$ (regular vector field) , a derivation of ${\mathcal O}_{V}$ extending $\partial$ corresponds to a regular section $s$ of a certain {\em shifted} tangent bundle $T_{\partial}(V)$ which we describe now. When $V$ is affine, say an algebraic subvariety of affine $n$-space, then $T_{\partial}(V)$ is the algebraic subvariety of affine $2n$-space defined by the equations  defining $V$, in indeterminates $x_{1},..,x_{n}$, together with  additional equations

$$\sum_{i=1}^{n}({\partial P}/\partial x_{i})({\bar x})u_{i} + P^{\partial}({\bar x})$$

 in indeterminates $x_{1},..,x_{n}, u_{1},..,u_{n}$, where $P$ ranges over the ideal $I(V)$ of $V$, and $P^{\partial}$ is obtained from $P$ by applying $\partial$ to the coefficients.  For a general algebraic variety $V$ over $F$, we can patch along affine charts to obtain $T_{\partial}(V)$. Note that if $V$ is defined over a field of constants then $T_{\partial}(V) = T(V)$, the tangent bundle of $V$.

Given a structure $(V,D_{V})$ on $V$ of an algebraic $\partial$-variety over $F$, we obtain a regular section $s$ (over $F$) of the projection $T_{\partial}(V) \to V$: Working in an affine chart, if $x_{1},..,x_{n}$ are the coordinate functions, and ${\bar a} = (a_{1},..,a_{n})\in V$, then
$s({\bar a})$ equals the evaluation of $(D_{V}(x_{1}),...,D_{V}(x_{n}))$ at ${\bar a}$.
 We sometimes write $(V,s)$ in place of $(V,D_{V})$.
 \newline
 (iii) Assuming $V$ to be an algebraic $\partial$-variety over $F$ where $V$ is (absolutely irreducible) and $F < K <{\mathcal U}$, then $D_{V}$ has a unique extension to a derivation of the structure sheaf of $V_{K}$ extending $\partial|K$. We sometimes identify this notationally with $D_{V}$. 
\end{rmk}

For $V$ a variety over $F$  as above, identifying $V$ with $V({\mathcal U})$ and working in an affine chart, given $x = (x_{1},..,x_{n})\in V$, then 
$(x,{\partial x})$ satisfies the equations defining $T_{\partial}(V)$. Hence can define $(V,s)^{\partial}$ to be $\{a\in V({\mathcal U}): \partial(a) = s(a)\}$. This is a {\em finite-dimensional differential algebraic variety}, whose order equals $dim(V)$. In fact any finite-dimensional definable set in ${\mathcal U}$ is {\em essentially} of the form $(V,s)^{\partial}$ for some algebraic variety $V$ over ${\mathcal U}$ and regular section $s:V\to T_{\partial}(V)$. 

\begin{rmk} Assume that $F = {\mathbb C}(t)$, and that $(V,s)$ is an absolutely irreducible algebraic $\partial$-variety over $F$. Then we can view $V$ as the generic fibre of some fibration $\pi:{\mathcal V} \to S$ where ${\mathcal V}$ is an irreducible complex algebraic variety, $S$ is $\mathbb C$ take away finitely many points, and $s$ becomes a regular vector field on ${\mathcal V}$ lifting the vector field $1$ on $S$. The differential polynomial equation $\partial(x) = s(x)$ translates to the differential equation 
$$dy/dt = s(y,t)$$
 on local analytic sections of $\pi$. If $D$ is a disc in $S$, and $K$ is a field of meromorphic functions on $D$ which embeds in ${\mathcal U}$ over ${\mathcal C}(t)$, then the points of $(V,s)^{\partial}$ in $K$ are solutions of the above differential equation.
\end{rmk}

Let us now specialise (for simplicity of notation and also because of the context of this paper) to the case where $V\subseteq {\mathcal U}^{n}$ is irreducible affine, so $(V,s)^{\partial}$ is a differential algebraic subvariety of ${\mathcal U}^n$. $F$ is still a differential field of definition for $(V,s)$.
\begin{defn} By a an algebraic ${\partial}$-subvariety $W$ of $(V,D_{V})$ (or the corresponding $(V,s)$) we mean a subvariety $W$ of $V$ (defined over some $K>F$) such that $s|W$ is a section of $T_{\partial}(W)\to W$. Equivalently $D_{V}$ (or rather its canonical extension to $K[V]$) preserves the ideal $I_{K}(W)$ of $W$. 
\end{defn}

The following is remarked in \cite{Hrushovski-Itai}, \cite{Buium} as well as \cite{Pillay-tworemarks}:
\begin{fct} The map taking a Kolchin closed subset $X$ of $(V,s)^{\partial}$ to its Zariski closure, established a bijection between the Kolchin closed subsets of $(V,s)^{\partial}$ and the algebraic $\partial$-subvarieties of $(V,s)$. (The inverse takes an algebraic $\partial$ subvariety $Y$ of $(V,s)$ to $Y\cap (V,s)^{\partial}$.) Moreover $X$ is irreducible as a Kolchin closed set iff its Zariski closure is irreducible as a Zariski closed set. 
\end{fct}

\begin{cor} $(V,s)^{\partial}$ is strongly minimal iff $V$ is positive-dimensional and $(V,s)$ has no proper (irreducible) positive-dimensional algebraic ${\partial}$-subvarieties.
\end{cor}
\begin{proof} This follows from Fact 2.16 and quantifier elimination for $DCF_{0}$. For example assuming the right hand side, Fact 2.16 implies that $(V,s)^{\partial}$ is infinite and has no proper infinite Kolchin closed sets. As any definable subset of $(V,s)^{\partial}$ is a finite Boolean combination of Kolchin closed sets we deduce strong minimality. Clearly it suffices to consider only irreducible Kolchin closed sets.
\end{proof}

The following special situation is relevant for the Painlev\'e equations: Let $(V,D_{V})$ be an algebraic $\partial$-variety over $F$, where $V$ is ${\mathbb A}^{2}$, affine $2$-space. So $D_{V}$ is simply a derivation of the polynomial ring $F[x,y]$ extending $\partial|F$, and for any $K>F$, $D_{V}\otimes_{\partial}K$ is the unique extension of $D_{V}$ to a derivation of $K[x,y]$ which extends $\partial|K$. We may often notationally identify $D_{V}\otimes_{\partial}K$ with $D_{V}$.

\begin{cor} (a) In the above situation ($V = {\mathbb A}^{2}$) $(V,D_{V})^{\partial}$ is strongly minimal if and only if there is no $K \geq F$ and nonconstant polynomial $P(x,y)$ over $K$
such that $D_{V}(P) = GP$ for some polynomial $G(x,y)$ over $K$ 
\newline
(b) Moreover these conditions are also equivalent to each of
\newline
- $(V,D_{V})$ has no $1$-dimensional algebraic $\partial$-subvariety,
\newline
-  $(V,D_{V})^{\partial}$ has no order $1$-definable subset. 
\end{cor}
\begin{proof} (a)Suppose first there IS nonconstant $P(x,y)$ over $K>F$ such that $D_{V}(P) = GP$ for some $G$. Let $I \subseteq K[x,y]$ be the ideal generated by $F$. It follows that $I$ is $D_{V}$-invariant. So the radical $\sqrt I$ of $I$ is also $D_{V}$-invariant (see 1.15 of \cite{Marker}). So $\sqrt I$ is the ideal of an algebraic $\partial$-subvariety $W$ of $V$ which has to be a curve. By Corollary 2.17, $(V,D_{V})^{\partial}$ is not strongly minimal. 

Conversely if $(V,D_{V})^{\partial}$ is not strongly minimal, then by Corollary 2.17 $(V,D_{V})$ has a proper positive-dimensional $\partial$-subvariety, which can be assumed to be irreducible (see above) hence is an irreducible plane curve $W$ defined over $K>F$ say. Then the ideal $I = I_{K}(W)$ is principal, generated by an irreducible polynomial $P(x,y)$ say. As $I$ is $D_{V}$-invariant it follows that $D_{V}(P) = GP$ for some $G$. 
\newline 
(b) is clear.

\end{proof}

 \begin{rmk} The right hand of Corollary 2.18 (a) is precisely Umemura's ``Condition J" on the vector field $D_{V}$. (See \cite{Umemura1}.)
 
 \end{rmk}

 \subsection{$\omega$-categoricity}
 
 A (infinite, one-sorted) structure $M$ in a countable language $L$ is said to be {\em $\omega$-categorical} if for each $n$ there are only finitely many $\emptyset$-definable subsets of $M$. The reason for the nomenclature is that $M$ is $\omega$-categorical if and only if $Th(M)$ (the complete first order theory of $M$) has exactly one countable model, up to isomorphism. For the rest of this section we recall some well-known facts, and use model-theoretic notions freely. 
 
 We would like to have the notion of  a definable (possibly with parameters) set $X$ in a structure $M$ being $\omega$-categorical. On the face of it this depends on a choice $A$ of some parameter set over which $X$ is defined, and we say that $X$ is $\omega$-categorical {\em in } M {\em over A} there are finitely many subsets of $X^{n}$ which are definable over $A$, for each $n$.  For models of an $\omega$-stable theory $T$ we get a robust notion.
 
 \begin{lem} Suppose $T$ is $\omega$-stable and $M$ is a model of $T$. Let $X\subseteq M^{n}$ be definable, and let $b,c$ be {\em finite} tuples from $N$ such that $X$ is definable over $b$ and also over $c$. Then $X$ is $\omega$-categorical in $M$ over $b$ iff $X$ is $\omega$-categorical in $M$ over $c$ (and in this case we just say $X$ is $\omega$-categorical in $M$). 
 \end{lem}
 \begin{proof} It is enough (by adding parameters) to prove that if $X$ is $\emptyset$-definable, and $\omega$-categorical over $\emptyset$ and $c$ is any finite tuple from $M$ then $X$ is $\omega$-categorical over $c$. Now by $\omega$-stability, $tp(c/X)$ is definable over an imaginary element $e\in X^{eq}$.  As $e$ is in the definable closure of a finite tuple from $X$  (and $\omega$-categoricity is preserved after naming a finite tuple from $X$), we see that $X$ is $\omega$-categorical over $e$, so also over $c$ (as every $c$-definable subset of $X^{m}$ is $e$-definable). 
 \end{proof} 
 
 \begin{lem} ($T$ $\omega$-stable.) Let $X$ be a strongly minimal definable set (in a model of an $\omega$-stable theory). Then $X$ is $\omega$-categorical if and only if for any finite tuple $b$ from $M$ over which $X$ is defined $acl(b)\cap X$ is finite. 
 \end{lem}
 \begin{proof} If $M$ is any structure then it is clear that $M$ is $\omega$-categorical just if for any finite tuple $a$ from $M$, there are only finitely many $a$-definable subsets of $M$. If $M$ is also strongly minimal and $a$ is a finite tuple from $M$, then as any $a$-definable subset of $M$
 is finite or cofinite, there are only finitely many $a$-definable subsets of $M$ iff $acl(a)$ is finite. So the Lemma holds for a structure $M$ in place of $X$. The full statement follows as in the proof of the previous Lemma.
 \end{proof}
 
 \begin{rmk} ($T$ $\omega$-stable.) If $X$ and $Y$ are nonorthogonal strongly minimal sets then $X$ is $\omega$-categorical iff $Y$ is $\omega$-categorical.
 
 \end{rmk}
 
 We now specialise to $T = DCF_{0}$, working in an ambient model which could be taken as ${\mathcal U}$. 
 \begin{lem} ($DCF_{0}$)  If $X$ is a definable, strongly minimal $\omega$-categorical set, then $X$ is geometrically trivial.
 \end{lem}
 \begin{proof} Firstly by a classical result of Zilber (see Theorem 4.17 of \cite{Pillay-book}) $X$ is modular. As remarked above, if $X$ is not geometrically trivial, $X$ is nonorthogonal to a strongly minimal group $G$ which is also $\omega$-categorical by Remark 2.22 as well as commutative. But $G$ definably embeds in an algebraic group $H$ hence (as characteristic is $0$), $G$ has only finitely many elements of any given finite order. This contradicts $\omega$-categoricity. 
 \end{proof}
 
 See \cite{Pillay-Fields} for an exposition of the following result of Hrushovski \cite{Hrushovski-Jouanalou}:
 \begin{fct} ($DCF_{0}$)  Let $X$ be strongly minimal and of order $1$ and orthogonal to the constants. Then $X$ is $\omega$-categorical (hence geometrically trivial by Lemma 2.23)
 \end{fct}
 
 It has been conjectured in a loose sense that all geometrically trivial strongly minimal sets in $DCF_{0}$ are $\omega$-categorical. The conjecture stated in the introduction that for distinct solutions $y_{1},..,y_{n}$ of a generic Painlev\'e equation, $y_{1},y_{1}',..,y_{n},y_{n}'$ are algebraically independent over ${\mathbb C}(t)$ is a strengthening (of course for the special case of generic Painlev\'e equations, which {\em are} strongly minimal as we see in the next section).

 \subsection{Definability of the $A\to A^{\sharp}$ functor} 
 
 Here our model-theoretic framework is the big model ${\mathcal U}$ of $DCF_{0}$. For the proof of our main result, we will need the following statement:
 \begin{lem} Let $\phi(x,y)$ be a  formula in the language of rings ($+, -, \cdot, 0, 1$), such that for each $b$, $\phi(x,b)$, if consistent, defines an abelian variety $A_{b}$. Then there is a formula $\psi(x,y)$ in our language $L$ of differential rings, such that for each $b$, $\psi(x,b)$ defines $(A_{b})^{\sharp}$. 
 \end{lem}
 
 So here $x$, $y$ are tuples of variables, and strictly speaking $\phi(x,y)$ is a pair of formulas, one for the underlying set (subset of a suitable projective space for example), and one for the graph of the group operation. The above result is implicit in Section 3.2 of \cite{Hrushovski-Itai}, and may even be needed for key results in that paper, but as there is neither an explicit statement or proof of Lemma 2.25 in \cite{Hrushovski-Itai} we take the opportunity to give a proof here. \\
 
 The first ingredient is a ``geometric" account of the construction of $A^{\sharp}$, appearing in many places including \cite{Hrushovski-Itai} and 
\cite{Marker-Quaderni} (but possibly originating with Buium). We summarize the situation. Let $A$ be an abelian variety over ${\mathcal U}$. Then $A$ has a ``universal vectorial 
extension" $\pi:{\tilde A} \to A$, that is $\tilde A$ is a commutative algebraic group which is an extension of $A$ by a vector group (power of the 
additive group) and for any other such extension $f:B\to A$, there is a unique homomorphism $g:{\tilde A}\to B$ such that everything commutes. (Here 
we work in the category of algebraic groups.) By functoriality of $T_{\partial}(-)$, for any algebraic group $G$ over ${\mathcal U}$, 
$T_{\partial}(G) \to G$ is a surjective homomorphism of algebraic groups, and moreover if $h:H\to G$ is a homomorphism of algebraic groups we obtain 
$T_{\partial}(h): T_{\partial}(H) \to T_{\partial}(G)$. In particular, taking $B$ to be $T_{\partial}({\tilde A})$, and $f:B \to A$ the composition of
$T_{\partial}(A)\to A$ with $T(\pi):T_{\partial}(\tilde A) \to T_{\partial}(A)$, we obtain a regular homomorphism $g:{\tilde A} \to T_{\partial}(A)$ which has to be a section of the canonical $T_{\partial}({\tilde A}) \to {\tilde A}$. We note that:
\begin{fct}
(i) $g:{\tilde A} \to T_{\partial}({\tilde A})$ is the {\em unique} regular homomorphic section of  $T_{\partial}({\tilde A}) \to \tilde A$, and, just for the record
\newline
(ii) $({\tilde A},g)$ is an algebraic $\partial$-group, in the obvious sense.
\end{fct}

From subsection 2.2, and Fact 2.26 (ii), we obtain $({\tilde A},g)^{\partial}$ which is now a finite-dimensional {\em differential algebraic group}, and the main point is:

\begin{fct} $A^{\sharp} = \pi(({\tilde A},g)^{\partial})$. 
\end{fct}

\vspace{2mm}
\noindent
The second ingredient is Lemma 3.8 of \cite{Hrushovski-Itai}, which we interpret as:
\begin{fct} The map which takes $A \to {\tilde A}$ is definable in $ACF$. Namely let $\phi(x,y)$, $\theta(y)$ be  formulas in the language of rings such that for all $b$ satisfying $\theta(y)$ (in some ambient algebraically closed field of characteristic $0$), $\phi(x,b)$ defines an abelian variety $A_{b}$. Then there is a formula $\chi(z,y)$ in the language of rings such that for all $b$ satisfying $\theta$, $\chi(z,b)$ defines $\widetilde {A_{b}}$ (and its canonical surjection $\pi_{b}$ to $A_{b}$). 
\end{fct}

\vspace{2mm}
\noindent
{\em Proof of Lemma 2.25}  
\newline
We prove the equivalent statement:
\newline
(*)  For any formula $\phi(x,y)$ in the language of rings and formula $\theta(y)$ in the language of differential rings such that for each $b\in {\mathcal U}$ satisfying $\theta(y)$, $\phi(x,b)$ defines an abelian variety $A_{b}$, then there is a formula $\psi(x,y)$ in the language of differential rings such that for each $b$ satisfying $\theta(y)$, $\psi(x,b)$ defines $(A_{b})^{\sharp}$.

\vspace{2mm}
\noindent
First let $\chi(z,y)$ be as in Fact 2.28. We prove (*) by induction on the Morley rank of $\theta(y)$. Suppose $RM(\theta(y)) = \alpha$. We may assume $\theta$ has Morley degree $1$. Let $b$ be a ``generic point" of $\psi(y)$ over $\emptyset$, namely $\models \neg\nu(b)$ for any formula $\nu(y)$ without parameters, of Morley rank $<\alpha$. Then $\chi(z,b)$ defines $\widetilde{A_{b}}$ and   $\pi_{b}: \widetilde{A_{b}}\to A_{b}$. Note that $T_{\partial}(\widetilde{A_{y}})$ and the canonical surjection $\lambda_{y}:T_{\partial}(\widetilde{A_{y}}) \to \widetilde{A_{y}}$ are also uniformly definable in $y$ in the differentially closed field: by formula $\eta(w,y)$ say (in the language of differential rings). 
Let $s: \widetilde{A_{b}} \to T_{\partial}(\widetilde{A_{b}})$ be the unique regular homomorphic section of $\lambda_{b}$ given by Fact 2.26. By uniqueness $s = s_{b}$ is definable (in the language of rings) over $b$, by a formula $\gamma(z,w,b)$ say. 
\newline
Now consider the formula $\theta'(y)$ which expresses:  $\theta(y)$ + ``$\gamma(z,w,y)$ defines a homomorphic section of the map $\lambda_{y}:T_{\partial}(\widetilde{A_{y}}) \to \widetilde{A_{y}}$". 
\newline
Note that  $\models \theta'(b)$. Moreover whenever $\models \theta'(c)$, then $\gamma(z,w,c)$ defines the {\em unique} regular homomorphic section, say $s_{c}: \widetilde{A_{c}} \to T_{\partial}(\widetilde{A_{c}})$. Hence by Fact 2.27, $(A_{c})^{\sharp}$ = $\pi_{c}((\widetilde{A_{c}},s_{c})^{\partial})$, and so is defined by a formula $\psi(x,c)$, where $\psi(x,y)$ does not depend on $c$. Hence (*) holds for $\theta'(y)$ in place of $\theta(y)$, but as $\theta'(y)$
is true of $b$, and $b$ is ``generic" for the Morley degree $1$ formula $\theta(y)$, the formula $\theta(y)\wedge\neg\theta'(y)$ has Morley rank $< \alpha$, and we can use induction.  This completes the proof of (*) and Lemma 2.25.

\section{Strong minimality and geometric triviality of the Painlev\'e equations}

For each of the six families of Painlev\'e equations we will prove strong minimality and geometric triviality for an equation with generic parameters, yielding the main result as stated in the abstract. For each of the families we will have to describe briefly relevant results by the ``Japanese school". 
Typically the equations are rewritten as (Hamiltonian) $\partial$-variety structures on ${\mathbb A}^{2}$, and a description of the transformations which preserve the family are given, culminating in a description of the complex parameters for which the Hamiltonian vector field satisfies Umemura's 
condition (J). We are interested only in the final conclusion. Recall that this condition (J) on the ``vector field" $D$ is: there is no $K \geq {\mathbb C}(t)$ and nonconstant polynomial $P(x,y)$ over $K$
such that $D(P) = GP$ for some polynomial $G(x,y)$ over $K$.

\subsection{The equation $P_{I}$.}  The equation is $y'' = 6y^{2} + t$. 

\begin{prop} The solution set of $P_{I}$ is strongly minimal and $\omega$-categorical (in fact disintegrated). 
\end{prop}

Strong minimality is given by the following (see Remark 2.3):

\begin{fct} Let $K < L$ be differential fields containing ${\mathbb C}(t)$. Let $y\in L$ be a solution of $y'' = 6y^{2} + t$. Then either $y\in K^{alg}$ or $tr.deg.(K\langle y \rangle/K) = 2$.
\end{fct}

Fact 3.2 is attributed to Kolchin in \cite{Marker} (Theorem 5.18) and to Kolchin-Kovacic in \cite{Umemura1} (Lemma 0). It was also rediscovered by Nishioka \cite{Nishioka1}. In any case \cite{Umemura1} gives a complete proof.

Painlev\'e proved that $P_{I}$ has no solution algebraic over ${\mathbb C}(t)$ and again a complete proof appears in \cite{Umemura1} (Lemma 0.8).

Nishioka \cite{Nishioka2}(Theorem 1)  proves:
\begin{fct} If $K < L$ are differential fields containing ${\mathbb C}(t)$ and $y_{1},..,y_{n}$ are distinct solutions of $P_{I}$ in $L$ each of which is not in $K^{alg}$, then $y_{1}, y_{1}',....y_{n},y_{n}'$ is algebraically independent over $K$. 
\end{fct}

This gives $\omega$-categoricity and geometric triviality of $P_{I}$. \\

Although not required we briefly mention the Hamiltonian nature of the algebraic $\partial$-variety attached to $P_{I}$. 
Rewrite $P_{I}$ as the system  $y' = x$, $x' = 6y^{2} + t$ in indeterminates $y,x$. 
Choosing $H(y,x,t)$ to be $\frac{1}{2}x^2-2y^3+ty$, we see that the system can be written in Hamiltonian form as

\begin{equation}
\begin{split}
 & y' =\frac{\partial H}{\partial x}\\
 & x' =-\frac{\partial H}{\partial y}
\end{split}
\end{equation}

We obtain the algebraic $\partial$-variety $({\mathbb A}^{2}, s)$ where $s(y,x) = (y,6y^{2} + t)$. Writing the vector field $s$  as a derivation $D$ on $K[y,x]$  (for some/any differential field $(K,\partial)\geq ({\mathbb C}(t),d/dt)$) extending $\partial$, we have that 

\[D=\partial+x\frac{\partial}{\partial y}+(6y^2+t)\frac{\partial}{\partial x}.\]

where for $P\in K[y,x]$, $\partial(P)$ is the result of applying $\partial$ to the coefficients of $P$. $D$ is sometimes called a ``Hamiltonian vector field".

The solution set of $P_{I}$ in $\mathcal U$ can be identified with  $({\mathbb A}^{2},D)^{\partial}$.

\subsection{The family $P_{II}$.}
For $\alpha\in \mathcal C$, $P_{II}(\alpha)$ is the following equation
\begin{equation*}
y''=2y^3+ty+\alpha.
\end{equation*}

Defining $x$ to be $y' + y^{2} +t/2$, we obtain the following equivalent (Hamiltonian) system:

\[S_{II}({\alpha})\left\{
\begin{array}{rll}
y'&=&x-y^2-\frac{t}{2}\\
x'&=&2xy+\alpha +\frac{1}{2}.\\
\end{array}\right.\]

We write $D(\alpha)$ for the corresponding vector field 
(as a derivation)  \[\partial+(x-y^{2} -\frac{t}{2})\frac{\partial}{\partial y}+(2xy + \alpha + \frac{1}{2})\frac{\partial}{\partial x}.\]

Note that the solution set of $P_{II}(-1/2)$ is not strongly minimal: the ODE $y' = -y^{2} - t/2$ defines a proper infinite differential algebraic subvariety.  Equivalently the curve defined by $x=0$ is an algebraic $\partial$-subvariety of $({\mathbb A}^{2}, D(-1/2))$.

In general there are ``Backlund transformations" which take solutions of $S_{II}(\alpha)$ to solutions of $S_{II}(-1-\alpha)$, $S_{II}(\alpha-1)$ and $S_{II}(\alpha + 1)$.  From (I) and (II) on p.160 of \cite{Umemura2} we see:

\begin{fct} For $\alpha\in\mathbb{C}$, $D(\alpha)$ satisfies condition (J)  if and only if 
$\alpha\notin\frac{1}{2}+\mathbb{Z}$.
\end{fct}

So from Corollary 2.18, we see:

\begin{cor}\label{cor1}
For $\alpha\in\mathbb{C}$, the solution set (in ${\mathcal U}$) of $P_{II}(\alpha)$ is strongly minimal if and only if $\alpha\not\in\frac{1}{2}+\mathbb{Z}$. Moreover if $\alpha\in\frac{1}{2}+\mathbb{Z}$ then the solution set of $P_{II}(\alpha)$ contains a definable subset of order $1$. 
\end{cor}

We now give the main result in the case of $P_{II}$. Its proof is the model for all subsequent proofs of the main result for $P_{III}- P_{VI}$,
\begin{prop} Let $\alpha\in {\mathbb C}$ be transcendental. Then (the solution set of) $P_{II}(\alpha)$ is strongly minimal and geometrically trivial.
\end{prop}
\begin{proof} Let $Y(\alpha)$ denote the solution set of $P_{II}(\alpha)$. Let $\alpha$ be transcendental. By Corollary 3.5, $Y(\alpha)$ is strongly minimal. As $ord(Y(\alpha)) = 2$, $Y(\alpha)$ is orthogonal to the differential algebraic variety ${\mathbb C}$ (see Remark 2.7). So if $Y(\alpha)$ is NOT geometrically trivial, then by the trichotomy theorem Theorem 2.8, $Y(\alpha)$ has to be nonorthogonal to $A^{\sharp}$ for some simple abelian variety $A$ with ${\mathbb C}$-trace $0$.
 
\vspace{2mm}
\noindent
{\em Claim I.}  $A$ is an elliptic curve.
\newline
{\em Proof.}  By Remark 2.7 $ord(A^{\sharp}) = 2$. So by 2.5 (ii), $dim(A) \leq 2$. If $dim(A) = 2$ then by 2.5(ii) again $A$ descends to $\mathbb C$, a contradiction, so $dim(A) = 1$ and $A$ is an elliptic curve. 

\vspace{2mm}
\noindent
So $A$ is the solution set of $y^{2} = x(x-1)(x-a)$ for some $a\in {\mathcal U}\setminus{\mathbb C}$ (in fact we could choose $
a\in {\mathbb C}(t)^{alg}$ but this does not simplify the argument). Let us rewrite $A$ as $E_{a}$. Applying Lemma 2.25 to the family of elliptic 
curves in Legendre form:  $\{E_{b}: y^{2} = x(x-1)(x-b): b\neq 0,1\}$, we obtain a formula $\psi(x,y,z)$ (in the language of differential rings) 
such that $\psi(x,y,b)$ defines $E_{b}^{\sharp}$  (for $b\neq 0,1$). Now the nonorthogonality of $Y({\alpha})$ and 
$E_{a}^{\sharp}$ is witnessed by some definable set $Z\subset Y(\alpha)\times E_{a}^{\sharp}$ which, without loss of 
generality projects onto each of $Y(\alpha)$, $E_{a}^{\sharp}$ and moreover such that each of these projections has 
fibres of cardinality $\leq k$ for some fixed $k$. Now $Z$ is defined by some formula $\chi(-,-,c)$ where we witness 
the parameters in the formula by $c$. 

\vspace{2mm}
\noindent
{\em Claim II.} There is an $L$-formula $\rho(w,u,v)$  with  additional parameter $t$ (where $u$ is possibly a tuple of variables) such that for any $\alpha_{1}, c_{1}, a_{1}$ from ${\mathcal U}$, ${\mathcal U} \models \rho(\alpha_{1},c_{1},a_{1})$ if and only $\alpha_{1}$ is a constant, $a_{1}$ is NOT a constant, and $\chi(-,-,c_{1})$ defines a subset of $Y(\alpha_{1})\times E_{a_{1}}^{\sharp}$ which projects onto each of $Y(\alpha_{1})$, $E_{a_{1}}^{\sharp}$ with all fibres of cardinality at most $k$. 
\newline
{\em Proof.}  This follows from the existence of the formula $\psi(x,y,z)$, i.e. uniform definablility of 
$E_{b}^{\sharp}$ as $b$ varies.  (The additional parameter $t$ in $\rho$ is there because $P_{II}(\alpha)$ has parameter $t$ in addition to $\alpha$.) 

\vspace{2mm}
\noindent
So by Claim II, we have that ${\mathcal U} \models \rho(\alpha, a, c)$, in particular taking $\eta(w)$ to be the formula $\exists u, v(\rho(w,u,v))$, we have that ${\mathcal U} \models \eta(\alpha)$. 

\vspace{2mm}
\noindent
{\em Claim III.}  $\models \eta(\alpha_{1})$ for all but finitely many $\alpha_{1}\in {\mathbb C}$. 
\newline
{\em Proof.} This is because $\eta(w)$ is over $t$, $\alpha$ is independent from $t$ over $\emptyset$ and ${\mathbb C}$ 
is strongly minimal. A little more slowly: By strong minimality of the field ${\mathcal C}_{\mathcal U} = {\mathbb C}$ 
of constants, $\eta(w)$ defines either a finite or cofinite subset  of ${\mathbb C}$. Suppose for the sake of 
contradiction that it defines a finite subset. Then as $\models \eta(\alpha)$, we would conclude that $\alpha\in acl(t)$ in ${\mathcal U}$. 
But 
$acl(t)$ is simply the field-theoretic algebraic closure of the differential field ${\mathbb Q}(t)$ generated by $t$, 
which clearly does not contain the transcendental element $\alpha \in {\mathbb C}$. So $\eta(w)$ defines a cofinite 
subset of ${\mathbb C}$ as required.

\vspace{2mm}
\noindent
By Claim III we conclude that $\models \eta(\alpha_{1})$ for some $\alpha_{1}\in 1/2 + {\mathbb Z}$. Hence by the definition of $\eta(w)$ there is $a_{1}\notin {\mathbb C}$ and some finite-to-finite definable  relation $R$ between $Y(\alpha_{1})$ and $E_{a_{1}}^{\sharp}$. 
By 3.5  $Y(\alpha_{1})$ contains an order $1$ definable subset $Z$ say. Let $K$ be a countable differential field over 
which all the data $Y(\alpha_{1})$, $E_{a_{1}}^{\sharp}, Z, R$ are defined. Let $z$ be a generic point of $Z$ over $K$
, and let $(x,y)$ a point of $E_{a_{1}}^{\sharp}$ such that $R(z,(x,y))$. Then $tr.deg(K\langle z\rangle/K) = 1$ (as 
$Z$ has order $1$ and $z\notin K^{alg}$), and $tr.deg(K\langle x,y\rangle /K) = 2$ (as by 2.5 $E_{a_{1}}^{\sharp}$ is 
strongly minimal of order $2$ and $(x,y)\notin K^{alg}$). We now have a contradiction, because $R$ witnesses that 
$(x,y)$ is in the algebraic closure of $K,z$, which we know to be the field-theoretic algebraic closure of $K\langle 
z\rangle$. 
\end{proof}

\subsection{The family $P_{III}$.}

On the face of it, the family $P_{III}$ is a $4$-parameter family: where $P_{III}(\alpha,\beta,\gamma,\delta)$, $\alpha,\beta,\gamma,\delta\in\mathbb{C}$ is given by the following 
\begin{equation*}
y''=\frac{1}{y}(y')^2-\frac{1}{t}y'+\frac{1}{t}(\alpha y^2+\beta)+\gamma y^3+\frac{\delta}{y}.
\end{equation*}

Okamoto \cite{Okam4} rewrites the equation as a $2$-parameter ``Hamiltonian system" which is then further studied by Umemura and Watanabe \cite{Umemura3}. We give a quick summary. Okamoto replaces $t^{2}$ by $t$ and 
$ty$ by $t$ to obtain the ``equivalent" family:

\[P_{III'}(\alpha,\beta,\gamma,\delta):\;\;\;\;\;  y''=\frac{1}{y}(y')^2-\frac{1}{t}y'+\frac{q^2}{4t^2}(\gamma q+\alpha)+\frac{\beta}{4t}+\frac{\delta}{4q} \]

and we are reduced to showing that for algebraically independent $\alpha, \beta, \gamma, \delta \in \mathbb C$, the solution set of $P_{III'}(\alpha, \beta, \gamma, \delta)$ is strongly minimal and geometrically trivial. In particular we can assume that neither $\gamma$ nor $\delta$ are $0$.\\

Okamoto then points out that for $\lambda, \mu\in \mathbb C$, the transformation taking $y$ to $\lambda y$ and $t$ to 
$\mu t$ takes $P_{III'}(\alpha,\beta,\gamma,\delta)$ to $P_{III'}(\lambda\alpha, \mu\lambda^{-1}\beta, 
\lambda^{2}\gamma, \mu^{2}\lambda^{-2}\delta)$. 

Hence taking $\lambda^{2} = 4/\gamma$ and $\mu^{2} = 1/\gamma\delta$  (assuming as above that $\gamma\delta\neq 0$), this transformation takes $P_{III'}(\alpha,\beta,\gamma,\delta)$ to $P_{III'}(\lambda\alpha, \mu\lambda^{-1}\beta, 4, -4)$, and moreover if 
$\alpha,\beta, \gamma, \delta$ are algebraically independent, so are $\lambda\alpha, \mu\lambda^{-1}\beta$. Hence the family  $P_{III'}(\alpha,\beta, \gamma, \delta)$ can be replaced by the family $P_{III'}(\alpha,\beta, 4,-4)$ and we are reduced to showing that for $\alpha,\beta$ algebraically independent the solution set of $P_{III'}(\alpha,\beta, 4, -4)$ is strongly minimal and geometrically trivial.   Finally $P_{III'}(\alpha,\beta,4,-4)$ can be written as a ``Hamiltonian $\partial$-variety"

\[S_{III'}({v_{1},v_{2}})\left\{
\begin{array}{rll}
\frac{dy}{dt}&=&\frac{1}{t}(2y^2x-y^2+v_1y+t)\\
\frac{dx}{dt}&=&\frac{1}{t}(-2yx^2+2yx-v_1x+\frac{1}{2}(v_1+v_2))

\end{array}\right.\]

where $\alpha = 4v_{2}$ and $\beta = -4(v_{1}-1)$. \\

(Note that $\alpha$ and $\beta$ are algebraically independent if and only if $v_{1}$ and $v_{2}$ are.)

The Hamiltonian is here: \[ H(v_{1},v_{2}) =\frac{1}{t}\left[y^2x^2-\left(y^2-v_{1}y-t\right)x-\frac{1}{2}(v_{1}+v_{2})y\right] \]

and the Hamiltonian vector field on ${\mathbb A}^{2}$ is:
\[D(v_{1},v_{2}) =  \partial + \frac{1}{t}(2y^2x-y^2+v_1y+t) \frac{\partial}{\partial y} + 
\frac{1}{t}(-2yx^2+2yx-v_1x+\frac{1}{2}(v_1+v_2))\frac{\partial}{\partial x} \]

From \cite{Umemura3} one now extracts:

\begin{fct}  $D(v_{1},v_{2})$ satisfies condition (J) if and only if  $v_{1} + v_{2}\notin 2{\mathbb Z}$ and $v_{1}-v_{2}\notin 2{\mathbb Z}$. 
\end{fct} 
{\em Commentary.}  This is contained in the statements of Theorem  1.2(i), (ii) , Proposition 2.1 and Corollary 2.5 of \cite{Umemura3}: 1.2(i) and (ii) say that if either $v_{1}+v_{2}$ or $v_{1}-v_{2}$ are in $2{\mathbb Z}$ then the solution set of the system $S_{III'}(v_{1},v_{2})$ has (many) proper differential algebraic subvarieties (hence by 2.18 $D(v_{1},v_{2})$ does not satisfy (J). On the other hand Proposition 2.1 says that if $D(v_{1},v_{2})$ satisfies (J) then for some integers $i,j,h\geq 0$, such that not both $i,j = 0$ such that 
$i(v_{1} + v_{2}) + j(v_{1}-v_{2}) + 2h(1-v_{1}) = 0$.  Strictly speaking this Proposition 2.1 studies the ``Hamiltonian vector field" corresponding to a different $\partial$-variety structure on ${\mathbb A}^{2}$, namely

\[\left\{
\begin{array}{rll}
\partial_{1}(y) &=& 2y^2x-y^2+v_1y+t\\
\partial_{1}(x) &=& -2yx^2+2yx-v_1x+\frac{1}{2}(v_1+v_2)

\end{array}\right.\]

where $\partial_{1}$ is the derivation $t\partial$ (or ``$t\frac{d}{dt}$ on ${\mathcal U}$ (with respect to which
${\mathcal U}$ is also saturated, differentially closed), but it is clear that $\partial$-subvarieties with respect to the original $\partial$-structure correspond to $\partial_{1}$-subvarieties with the new $\partial_{1}$-structure.

\begin{cor} The solution set of $S_{III'}(v_{1},v_{2})$ is strongly minimal if and only if $v_{1}+ v_{2}\notin 2{\mathbb Z}$ and $v_{1}-v_{2}\notin 2{\mathbb Z}$. 
\end{cor}

We conclude our main result for $P_{III}$:

\begin{prop} Let $v_{1}, v_{2}$ be algebraically independent. Then (the solution set of) $S_{III}'(v_{1},v_{2})$ is strongly minimal and geometrically trivial. 
\end{prop} 
\begin{proof} Let $Y(v_{1},v_{2})$ denote the solution set of  $S_{III}'(v_{1},v_{2})$. Strong minimality is by Corollary 3.8. Suppose that geometric triviality fails. We copy the proof of Proposition 3.6. So $Y(v_{1},v_{2})$ is nonorthogonal to $E_{a}^{\sharp}$ for an elliptic curve $E_{a}: y^{2} = x(x-1)(x-a)$ with $a\notin {\mathbb C}$.  Using definability of the family $E_{b}^{\sharp}$, we express nonorthogonality of $Y(v_{1},v_{2})$ to some $E_{b}^{\sharp}$ by a formula $\theta(v_{1},v_{2},t)$. As  $v_{2}$ is independent from $v_{1},t$, there are cofinitely many $v\in {\mathbb C}$ such that $\theta(v_{1},v,t)$ holds. In particular we can find such $v\in v_{1} + 2{\mathbb Z}$, and we have a contradiction as in 3.6 (using now Fact 3.7)

\end{proof}

\subsection{The family $P_{IV}$.}

$P_{IV}(\alpha,\beta)$, $\alpha,\beta\in\mathbb{C}$ is given by the following equation
\begin{equation*}
y''=\frac{1}{2y}(y')^2+\frac{3}{2}y^3+4ty^2+2(t^2-\alpha)y+\frac{\beta}{y}
\end{equation*}

Following \cite{Okam3} and \cite{Umemura2} the equations can be rewritten in the form

\[S_{IV}(v_{1},v_{2},v_{3})\left\{
\begin{array}{rll}
y' &=&2xy-y^2-2ty+2(v_1-v_2)\\
x' &=&2xy-x^2+2tx+2(v_1-v_3)
\end{array}\right.\]

where $v_{1},v_{2},v_{3}$ are complex numbers satisfying $v_{1} + v_{2} + v_{3} = 0$. 
\newline
(*) Here $\alpha = 3v_{3} + 1$ and $\beta = -2(v_{2}-v_{1})^{2}$. \\

Let $\mathcal{V}$ be the plane defined by $v_{1} + v_{2} + v_{3}$. So the algebraic independence of $\alpha$ and $\beta$ corresponds to ${\bf v}$ being a generic point on $\mathcal V$. Let $D({\bf v})$ be the derivation (Hamiltonian vector field) corresponding to $S_{IV}(\bf v)$. 
Let $\mathcal{W}=\left\{{\bf v}\in\mathcal{V}:\;v_1-v_2\in\mathbb{Z}\right\}\cup \left\{{\bf v}\in\mathcal{V}:\;v_2-v_3\in\mathbb{Z}\right\}\cup \left\{{\bf v}\in\mathcal{V}:\;v_3-v_1\in\mathbb{Z}\right\}$. Then the analysis in   \cite{Umemura2}, specifically Theorem 3.2, Proposition 3.5, and Corollary 3.9 yields

\begin{fct} For ${\bf v}\in\mathcal{V}$, $D({\bf v})$ satisfies condition (J) if and only if ${\bf v}\notin\mathcal{W}$.\\
 
\end{fct}
\begin{cor}\label{CorollaryP4}
For ${\bf v}\in\mathcal{V}$, the solution set of $P_{IV}({\bf v})$ is strongly minimal if and only if ${\bf v}\not\in\mathcal{W}$.
\end{cor}

\begin{prop} Let $Y(\alpha, \beta)$ denote the solution set of $P_{IV}(\alpha,\beta)$. Suppose $\alpha, \beta \in {\mathbb C}$ are algebraically independent. Then $Y(\alpha,\beta)$ is strongly minimal and geometrically trivial.
\end{prop}
\begin{proof} Let $(v_{1},v_{2},v_{3}) \in {\mathcal V}$ correspond to $(\alpha,\beta)$ as in (*) above. So $v_{1}$ and $v_{2}$ are algebraically independent and $v_{3} = -v_{1} - v_{2}$. It suffices to work with the solution set $Y({\bf v})$ of $S_{IV}({\bf v})$. Strong minimality is by Corollary 3.11. Again if geometric triviality fails this is witnessed by the truth of a formula $\theta(v_{1},v_{2},t)$, where $\theta(u,w,t)$ expresses  
the existence of a differential algebraic correspondence between $Y(u,w, -u-w)$ and some $E_{b}^{\sharp}$ with $b\notin {\mathbb C}$. The independence of $v_{2}$ over $v_{1},t$ implies that $\models \theta(v_{1},w,t)$ for all but finitely many $w \in {\mathbb C}$. So we can find such $w\in v_{1} + {\mathbb Z}$ giving a contradiction to Corollary 3.11.

\end{proof}

\subsection{The family $P_{V}$}
We recall that $P_{V}(\alpha,\beta,\gamma,\delta)$, $\alpha,\beta,\gamma,\delta \in \mathbb{C}$, is given by the following equation
\begin{equation*}
y'' =\left(\frac{1}{2y}+\frac{1}{y-1}\right)(y')^2-\frac{1}{t}y'+\frac{(y-1)^2}{t^2}\left(\alpha y+\frac{\beta}{y}\right)+\gamma\frac{y}{t}+\delta\frac{y(y+1)}{y-1} 
\end{equation*}

The situation is similar to the case of $P_{III}$ above, in that, in the light of certain transformations preserving the equation, $P_{V}$ can be written as a  $3$-parameter family. 
The analysis is carried out by Okamoto\cite{Okam2} and then Watanabe\cite{Watanabe} and again we give a summary:

First $P_{V}(\alpha,\beta, \gamma, \delta)$ can be rewritten as the Hamiltonian system  

\[ y' = \frac{\partial H}{\partial x} \\, 
x'  = \frac{\partial H}{\partial y} \] 

 where the Hamiltonian polynomial is

\[ H_{V}(\kappa_0,\kappa_1,\theta,\eta):\;\;\; H=\frac{1}{t}\left[y(y-1)^2x^2-\left(\kappa_0(y-1)^2+\theta y(y-1)+\eta 
ty\right)x+\kappa(y-1)\right] \]
and 
\[\alpha=\frac{1}{2}\kappa_1^2,\;\;\; \beta=-\frac{1}{2}\kappa_0^2,\;\;\;\gamma=-\eta(\theta+1),\;\;\;\delta=-\frac{1}{2}\eta^2,\;\;\;\kappa=\frac{1}{4}(\kappa_0+\theta)^2-\frac{1}{4}\kappa_1^2.\]

Okamoto points out that the transformation $(y,x,H,t) \to (y,x,\lambda H, \lambda^{-1}t)$ and $\eta \to \lambda^{-1}\eta$ ``preserves the system" for any nonzero $\lambda$. As we may assume $\delta\neq 0$ (so also $\eta\neq 0$) we may choose $\lambda = \eta$, and hence the transformation takes $\delta$ to $-1/2$ (and $\eta$ to $1$).  Hence we need only consider the system 
$P_V(\alpha,\beta,\gamma,-\frac{1}{2})$, and prove that for $\alpha, \beta, \gamma$ algebraically independent, the solution set of $P_{V}(\alpha, \beta, \gamma, -\frac{1}{2})$ is strongly minimal and geometrically trivial. \\

Let us now define: 
\[v_1=-\frac{1}{4}(2\kappa_0+\theta),\;\;\;v_2=\frac{1}{4}(2\kappa_0-\theta),\;\;\;v_3=\frac{1}{4}(2\kappa_1+\theta),\;\;\;v_4=-\frac{1}{4}(2\kappa_1-\theta),\]
then the vector ${\bf v}=(v_1,v_2,v_3,v_4)$ is on the complex hyperplane $\mathcal{V}$ in $\mathbb{C}^4$ defined by $v_1+v_2+v_3+v_4=0$. So our family is now parameterized by $\mathcal{V}$ \\

Secondly make the following substitution: replace $y(y-1)^{-1}$ by $y$, and $-y(y-1)^{2}x + (v_{3}-v_{1})(y-1)$.\\

Then our system can be written as:  


\[S_{V}({\bf v})\left\{
\begin{array}{rll}
y' &=&\frac{1}{t}(2y^2x-2yx+ty^2-ty+(v_1-v_2-v_3+v_4)y+v_2-v_1)\\

x'&=& \frac{1}{t}(-2yx^2+x^2-2txy+tx-(v_1-v_2-v_3+v_4)x+(v_3-v_1)t)

\end{array}\right.\]

and note that $\alpha = \frac{1}{2}(v_3-v_4)^2$,
$\beta =-\frac{1}{2}(v_2-v_1)^2$ and 
$\gamma = 2v_1+2v_2-1$.
\newline
Let $D({\bf v})$ the corresponding (Hamiltonian) vector field $\partial +  
(\frac{1}{t}(2x^2x-2yx+ty^2-ty+(v_1-v_2-v_3+v_4)y+v_2-v_1))\frac{\partial}{\partial y} + 
(\frac{1}{t}(-2yx^2+x^2-2txy+tx-(v_1-v_2-v_3+v_4)x+(v_3-v_1)t)\frac{\partial}{\partial x}$\\


Hence we want to prove that for generic ${\bf v}\in \mathcal V$, $S_{V}({\bf v})$ is strongly minimal and geometrically trivial. 

Consider the union of  lines in $\mathcal{V}$ given by 
\begin{eqnarray*}
\mathcal{W}&=\left\{v\in\mathcal{V}:\;v_1-v_2\in\mathbb{Z}\right\}\cup \left\{v\in\mathcal{V}:\;v_1-v_3\in\mathbb{Z}\right\}\cup \left\{v\in\mathcal{V}:\;v_1-v_4\in\mathbb{Z}\right\}\\
&\cup \left\{v\in\mathcal{V}:\;v_1-v_3\in\mathbb{Z}\right\}\cup \left\{v\in\mathcal{V}:\;v_2-v_4\in\mathbb{Z}\right\}\cup \left\{v\in\mathcal{V}:\;v_3-v_4\in\mathbb{Z}\right\},
\end{eqnarray*}
The following is proved in  \cite{Watanabe} (Theorem 1.2, Proposition 2.1, and Corollary 2.6):

\begin{fct} For ${\bf v}\in\mathcal{V}$, $D({\bf v})$ satisfies condition (J) if and only if ${\bf v}\notin\mathcal{W}$. 
\end{fct}
{\em Commentary.} Strictly speaking Watanabe considers rather the vector field corresponding to the $\partial_{1}$-variety structure on 
${\mathbb A}^{2}$:
\[\left\{
\begin{array}{rll}
\partial_{1}(y) &=& 2x^2x-2yx+ty^2-ty+(v_1-v_2-v_3+v_4)y+v_2-v_1)\\
\partial_{1}(x) &=& -2yx^2+x^2-2txy+tx-(v_1-v_2-v_3+v_4)x+(v_3-v_1)t)
\end{array}\right.\]
where $\partial_{1} = t\partial$.

But as in the case of $P_{III}$ above, this suffices.

\begin{cor}\label{CorollaryP5}
For ${\bf v}\in\mathcal{V}$, the solution set of $S_{V}({\bf v})$  is strongly minimal if and only if ${\bf v}\not\in\mathcal{W}$.
\end{cor}

We conclude, as previously:
\begin{prop} For $v_{1}, v_{2}, v_{3} \in {\mathbb C}$ algebraically independent, the solution set of $S_{V}({\bf v})$ is strongly minimal and 
geometrically trivial. Hence by the reductions above, for algebraically independent $\alpha, \beta, \gamma, \delta \in {\mathbb C}$ the solution set of $P_{V}(\alpha, \beta, \gamma, \delta)$ is strongly minimal and geometrically trivial.

\end{prop}

\subsection{The family $P_{VI}$}

Recall that $P_{VI}(\alpha,\beta,\gamma,\delta)$, $\alpha,\beta,\gamma,\delta \in \mathbb{C}$, is given by the following equation
\begin{equation*}
\begin{split}
y''=&\frac{1}{2}\left(\frac{1}{y}+\frac{1}{y+1}+\frac{1}{y-t}\right)(y')^2-\left(\frac{1}{t}+\frac{1}{t-1}+\frac{1}{y-t}\right)y'\\
&+\frac{y(y-1)(y-t)}{t^2(t-1)^2}\left(\alpha+\beta\frac{t}{y^2}+\gamma\frac{t-1}{(y-1)^2}+\delta\frac{t(t-1)}{(y-t)^2}\right)
\end{split}
\end{equation*}

The equation is studied by Okamoto \cite{Okam1} followed by Watanabe \cite{Watanabe1}. Again we summarise what we need. 


Given $\alpha, \beta, \gamma, \delta\in {\mathbb C}$ let $a_{1}, a_{2}, a_{3}, a_{4}$ be complex numbers such that 
$\alpha =\frac{1}{2}(a_1-a_2)^2$, 
$\beta = -\frac{1}{2}(a_3+a_4)^2$, 
$\gamma = \frac{1}{2}(a_3-a_4)^2$, and
$\delta = -\frac{1}{2}(1-(1-a_1-a_2)^2)$. 

Then the equation $P_{VI}(\alpha, \beta, \gamma, \delta)$ can be written as the (Hamiltonian) system:

\[S_{VI}({\bf a})\left\{
\begin{array}{rll}
y' &=& \frac{1}{t(t-1)}(2y(y-1)(y-t)x+(a_1+a_2-2a_3)y^2\\
& &+(2a_3t-a_1-a_2+a_3+a_4)y-(a_3+a_4)t)\\
x'&=& \frac{1}{t(t-1)}(-3y^2x^2+2(1+t)yx^2-tx^2-2(a_1+a_2-2a_3)yx\\
& &-(2a_3t-a_1-a_2+a_3+a_4)x-(a_3-a_2)(a_3-a_1))

\end{array}\right.\]

and where $D({\bf a}) = \partial + (\frac{1}{t(t-1)}(2y(y-1)(y-t)x+(a_1+a_2-2a_3)y^2 +(2a_3t-a_1-a_2+a_3+a_4)y-(a_3+a_4)t))\frac{\partial}{\partial y} + (\frac{1}{t(t-1)}(-3y^2x^2+2(1+t)yx^2-tx^2-2(a_1+a_2-2a_3)yx
-(2a_3t-a_1-a_2+a_3+a_4)x-(a_3-a_2)(a_3-a_1))\frac{\partial}{\partial x}$ is the derivation giving the corresponding $\partial$-variety structure on ${\mathbb A}^{2}$. \\

Let ${\mathcal M}$ be the collection of $(a_{1},a_{2},a_{3},a_{4}) \in {\mathbb C}^{4}$ such that for some $1\leq i < j \leq 4$, $\pm a_{i} \pm a_{j} \in {\mathbb Z}$. 

Then the following is proved in \cite{Watanabe1} (Theorem 2.2, Proposition 3.1, and Corollary 3.7). 

\begin{fct} For ${\bf a} \in {\mathbb C}^{4}$, $D({\bf a})$ satisfies condition (J) if and only if ${\bf a}\notin {\mathcal M}$.
\end{fct}
{\em Commentary.}  Again Watanabe works instead with  $t(t-1)D({\bf a})$, but it suffices.

\begin{cor} The solution set of $S_{VI}({\bf a})$ is strongly minimal if and only if ${\bf a}\notin {\mathcal M}$.
\end{cor}

\begin{prop} If $\alpha, \beta, \gamma, \delta \in {\mathbb C}$ are algebraically independent (and transcendental), then the solution set of $P_{VI}(\alpha, \beta, \gamma, \delta)$ is strongly minimal and geometrically trivial.
\end{prop}
\begin{proof} It is enough to work with the solution set of the system $S_{VI}({\bf a})$ and note that $a_{1}, a_{2}, a_{3}, a_{4}$ are algebraically 
independent, so by Corollary its solution set $Y(a_{1},a_{2},a_{3},a_{4})$ is strongly minimal. If it is not geometrically trivial, then by indepence 
of $a_{1}, a_{2}, a_{3},a_{4}, t$ in $DCF_{0}$, for all but finitely many $c\in {\mathbb C}$, $Y(c,a_{2},a_{3},a_{4})$ is in finite-to finite 
correspondence with some $E_{b}^{\sharp}$ with $b\notin {\mathbb C}$. Choosing such $c\in a_{2}+ {\mathbb Z}$ gives a contradiction to Corollary 3.17

\end{proof}

\subsection{Further Remarks}  First note that the methods above show that for the families $P_{III}$, $P_{V}$ and  $P_{VI}$, the solution set of an equation is (strongly minimal and) geometrically trivial as long as  the transcendence degree of the tuple $\alpha, \beta, \gamma, \delta$ is at least $2$. 

We would also guess that in fact the solution set of any of the equations is geometrically trivial whenever it is strongly minimal, but this is not given by our methods.

Note that the model-theoretic content of our methods is:

\begin{prop} Let $y'' = f(y,y', b_{1},..,b_{n})$ where $f(-,-,-..)$ is a rational function (over $\mathbb Q$) and $b_{1},..,b_{n}\in {\mathcal U}$. 
Suppose that the solution set of the equation is strongly minimal and {\em not} geometrically trivial. Then there is a formula $\theta(y_{1},..,y_{n})$(without parameters)  true of $b_{1},..,b_{n}$ such that for any $c_{1},..,c_{n}$ satisfying $\theta$ the solution set of $y'' = f(y,y',c_{1},..,c_{n})$ is strongly minimal (and nontrivial).
\end{prop}

Likewise if $X({\bar b})$ is a strongly minimal set of order $2$ which is defined over ${\bar b}$ and is not geometrically trivial, then there will be a formula $\theta({\bar y})$ true of ${\bar b}$ such that for any ${\bar c}$ satisfying $\theta$, $X_{\bar c}$ has Morley rank $1$. \\

The next remark concerns ``downward semi-definability of Morley rank", as discussed in \cite{Hrushovski-Scanlon}. The point is that each of the 
Painlev\'e families witnesses non downward semi-definability of Morley rank. For example for $P_{II}(\alpha)$ the analysis of \cite{Umemura2} that we cite also gives that for $\alpha\in 1/2 + {\mathbb Z}$, the solution set of $P_{II}(\alpha)$ contains infinitely many (in fact a definable family of) order $1$ definable sets, hence has Morley rank $2$. So:
\begin{fct}
For $\alpha\in {\mathbb C}$ transcendental, the solution set $Y(\alpha)$ of $P_{II}(\alpha)$ has Morley rank $1$, but for every formula $\theta(x)$ 
true of $\alpha$ there is $\alpha_{1}$ satisfying $\theta(x)$ such that $Y_{\alpha_{1}}$ has Morley rank $>1$, 
\end{fct}

This fact, together with the proof of Lemma 1.1 in \cite{Hrushovski-Scanlon}, shows that there is a $\emptyset$-definable set $X$ of order $3$ in $DCF_{0}$ with  Morley rank different from Lascar rank. This answers (in a negative sense) Question 2.9 of \cite{Hrushovski-Scanlon}. 

\vspace{2mm}
\noindent 
Finally one might wonder whether our soft proofs of geometric triviality of certain order $2$ equations might apply to higher order equations. But an obstacle is precisely non downward semi-definability of Morley rank for more general families $A^{\sharp}$, as given in \cite{Hrushovski-Scanlon}: There is a definable (in $DCF_{0}$) family  $(A_{b}:b\in B)$ (with $B$ of finite Morley rank) family of $2$-dimensional abelian varieties such for generic $b\in B$, $A_{b}$ is simple (hence $A_{b}^{\sharp}$ is strongly minimal), but for every $\emptyset$-definable subset $B'\subset B$ there is $b'\in B'$ such that $A_{b'}$ is not simple (hence $A_{b'}^{\sharp}$ has Morley rank $2$).

\appendix

\section{Irreducibility}

The expression ``irreducibility" here refers to irreducibility of a function to ``known" functions. The context is the attempt to discover the 
second order ODE's with the Painlev\'e property which ``define", via their solutions, really new functions. We did not read the original article but 
we understand that in \cite{Pain}, Painlev\'e initiated a definition of ``known function". This was taken up again the 1980's by Umemura and 
others who tried to give clear definitions of a ``known" (or ``classical") function and an  ``irreducible" ODE, and thus open up the possibility of a rigorous proof that the list $P_{I}-P_{VI}$ really define new special functions. The results of the Japanese school which we use and refer to in the body of this paper were really aimed at proving ``irreducibility". This also figures in work by Casale for example \cite{Casale} which tries to prove irreducibility via Malgrange's Galois groupoid.\\

Here we aim to describe this theory of irreducibility and give, without detailed proofs, some model-theoretic account (although an equivalent differential algebraic account in terms of strongly normal extensions is well-known)
There are various levels of abstraction at which one can present these ideas of ``classical function" and ``irreducibility", but we will roughly follow the presentation in \cite{Umemura1}. \\

Let $D$ be a domain in $\mathbb C$ and  ${\mathcal F}(D)$ denote the {\em differential} field of meromorphic functions on $D$, equipped with the 
derivation $d/dt$. Note that ${\mathbb C}(t)$ is a subdifferential field. We will consider only functions in ${\mathcal F}(D)$ where $D$ may vary. We 
also may identify a function $f\in  {\mathcal F}(D)$ with its restriction to a smaller domain $D'$. We need also the notion of the logarithmic derivative $\partial ln_{G}$ corresponding to a connected complex algebraic group $G$. Let $TG$ be the tangent bundle of $G$, also a connected complex algebraic group, a semidirect product of $G$ with $LG = {\mathbb C}^{n}$. Let $F:D \to G$ be holomorphic. Then for $t\in G$, $F'(t) = dF/dt$ can be identified with a point in $TG$ in the fibre of $F(t)$ and then $F'(t)F(t)^{-1}$ (multiplication in the sense of $TG$) lies in $LG$, and we define $(\partial ln_{G}(F))$ to be the holomorphic function from $D$ to $LG$ whose value at $t$ is precisely $F'(t)F(t)^{-1}$. Likewise if $F:D\to G$ is meromorphic $\partial ln_{G}(F): D \to LG$ is meromorphic.\\

This logarithmic derivative $\partial ln_{G}$ in fact  makes sense in any differential field, in particular in our universal differentially closed field ${\mathcal U}$, as a differential rational crossed homomorphism from $G({\mathcal U})$ to $LG({\mathcal U})$:  for $g\in G({\mathcal U})$, $\partial ln_{G}(g) = \partial g \cdot g^{-1}$ where $\cdot$ is multiplication in $TG({\mathcal U})$. \\

We will denote our base differential field ${\mathbb C}(t)$ by $K$.

\begin{defn} (a) Following Umemura \cite{Umemura1} w give an inductive definition of a {\em classical function}.
\newline
(i) any $f\in K$ is classical.
\newline
(ii) Suppose that $f_{1},..,f_{n}\in {\mathcal F}(D)$ are classical, and $f\in {\mathcal F}(D')$ (for some appropriate $D'\subseteq D$) is in the algebraic closure of the differential field generated by ${\mathcal C}(t)(f_{1},..,f_{n})$, then $f$ is classical.
\newline
(iii) Let $G$ be a connected complex algebraic group with $LG = {\mathbb C}^{n}$. Suppose $f_{1},..,f_{n}\in {\mathcal F}(D)$ are classical, and for some $D'\subseteq D$, $f: D'\to G$ is a meromorphic function such that $\partial ln_{G}(f) = (f_{1},..,f_{n})$. Then the coordinate functions of $f$ are classical. 
\newline
(b) An equation $f(y,y',...y^{(n)}) = 0$ (with coefficients from $K$ is then said to be {\em irreducible} (with respect to classical functions) if no solution is classical (in particular there are no algebraic solutions).
\end{defn}

\begin{rmk}  If $F$ is any differential field (with algebraically closed field of constants ${\mathcal C}_{F}$), $G$ a connected algebraic group over ${\mathcal C}_{F}$, $a\in LG(F)$, and $g\in G(L)$ ($L>F$) is such that $\partial ln_{G}(g) = a$, and moreover ${\mathcal C}_{F} = {\mathcal C}_{F\langle g\rangle}$ then $F\langle g\rangle$ is a {\em strongly normal extension} of $F$ (and for $F$ algebraically closed any strongly normal extension of $F$ has this form for some $G$, $a$, $g$). In particular as is well-known $F\langle g\rangle$ is contained in some differential closure $F^{diff}$ of $F$, hence $tp(g/F)$ is isolated (or constrained in the sense of Kolchin).
\end{rmk}

As mentioned earlier any differential subfield of any ${\mathcal F}(D)$ which is countably generated over ${\mathbb C}(t)$ can be embedded (over ${\mathbb C}(t)$) in our universal differential field ${\mathcal U}$, hence any $f\in {\mathcal F}(D)$ can be assumed to be an element of 
${\mathcal U}$. 

\begin{prop} A function $f\in {\mathcal F}(D)$ is classical if and only if in the structure ${\mathcal U}$, $tp(f/{\mathbb C}(t))$ is isolated and analysable in the constants ${\mathcal C}_{\mathcal U} = {\mathbb C}$. 
\end{prop} 
 
We recall the notion of analyzability and mention some ingredients for the (easy) proof of Proposition A.3.\\

 Analysability of $tp(f/K)$ in the constants means that (possibly passing to a larger universal differentially closed field ${\mathcal U}_{1}$), there are  $a_{0},a_{1},..,a_{n}$ such that $f\in K\langle a_{0},a_{1},..,a_{n}\rangle^{alg}$ (i.e. $f\in acl(K,a_{0},..,a_{n})$) and for each $i$, either $a_{i}\in K\langle  
a_{0},..,a_{i-1}\rangle ^{alg}$ or $tp(a_{i}/K\langle a_{0},..,a_{i-1}\rangle$ is {\em stationary} and internal to the constants. Moreover if such $a_{0},..,a_{n}$ exist then we may choose them (maybe enlarging the sequence) to be in $dcl(K,f)$ (namely in $K\langle f\rangle$). 
Stationarity of $tp(a/F)$ means that the differential algebraic variety over $F$ of which $a$ is the generic point is {\em absolutely irreducible}. Internality of $tp(a/F)$ to the constants means that  there is a differential field $L>F$ such that $a$ is independent from $L$ over $F$ and $a\in L\langle c_{1},..,c_{k}\rangle$ for some {\em constants} $c_{1},..,c_{k}$.
Analysability of $tp(a/K)$ in the constants is equivalent to every extension of this type being nonorthogonal to the constants.\\

\begin{rmk} (e.g. Section 18.3 of \cite{Poizat}) Let $F$ with a differential field with algebraically closed field of constants, and  $a$ a tuple from ${\mathcal U}$, then $tp(a/F)$ is isolated and internal to the constants if and only if $a$ is contained in some strongly normal extension $L$ of $F$. 
\end{rmk}

Putting together Remarks A.2 and A.4 the proof of Proposition A.3 is easy. \\

As far as irreducibility is concerned:
\begin{rmk} Let $f(y,y',..,y^{(n)}) = 0$ (with $n\geq 2$) be an ODE over $K = {\mathbb C}(t)$, which has no algebraic (over $K$) solutions and whose solution set in ${\mathcal U}$ is strongly minimal. Then $f(y,y',..,y^{(n)}) = 0$ is irreducible with respect to classical functions. 
\end{rmk}
\begin{proof} As the order is $\geq 2$, the solution set of the equation is orthogonal to the constants, hence together with no algebraic solutions we obtain irreducibility. 
\end{proof}

In  \cite{Umemura1}, Umemura considers an extension the inductive definition A.1.  We again consider functions in ${\mathcal F}(D)$ for varying $D$.
\begin{defn} (a) We give an inductive definition of a $1$-classical function:
\newline 
Clauses (i)- (iii) are exactly as in Definition A.1 with $1$-classical in place of classical.
\newline
(iv) Suppose $f_{1},..,f_{n} \in {\mathcal F}(D)$ are $1$-classical, and $f \in {\mathcal F}(D)$ is a solution of an $ODE$ $f(y,y')= 0$ where $f$ has coefficients from $K(f_{1},..,f_{n})$. Then $f$ is $1$-classical.
\newline
(b) An $ODE$ over $K$ is said to be irreducible with respect to $1$-classical functions if no solution of it is $1$-classical. 
\end{defn}

We will state without proof analogues of Proposition A.3 and Remark A.5  The first depends on Fact 2.24 among other things.

\begin{prop} A function $f\in {\mathcal F}(D)$ is $1$-classical if and only if $tp(f/K)$ is isolated and analysable in the constants together with the family of all order $1$ equations $f(y,y') = 0$ (with coefficients from ${\mathcal U}$).
\end{prop} 

\begin{prop} An $ODE$  $y'' = f(y,y')$ over $K$ is irreducible with respect to $1$-classical functions if and only if it has no algebraic solutions and its solution set is strongly minimal.
\end{prop}


\end{document}